\documentclass[12pt,article]{amsart}
\usepackage{amssymb}
\usepackage{amsfonts}
\usepackage{amsbsy,url}
\usepackage{latexsym}
\usepackage{amssymb,latexsym,amsmath,amsthm}
\usepackage[normalem]{ulem}
\usepackage{color}
\usepackage{soul}
\theoremstyle{plain}

\theoremstyle{remark} 

\newtheorem {theo} {\bf Theorem} [section]
\newtheorem {prop} [theo] {\bf Proposition}
\newtheorem {coro} [theo] {\bf Corollary}
\newtheorem {lem} [theo] {\bf Lemma}
\newtheorem {defi} {\bf Definition}[section]

\newtheorem{rem}{\bf Remark}[section]
\newtheorem{conjecture}{\bf Conjecture}[section]

\numberwithin{equation}{section}
\makeatletter
\@namedef{subjclassname@2020}{%
  \textup{2020} Mathematics Subject Classification}
\makeatother
\definecolor{vbcolor}{rgb}{0.95, 0.15, 0.6}

\begin{document}
\title{Dynamical Frames and Hyperinvariant Subspaces}
\author{V. Bailey}
\address{Department of Mathematics\\
University of Oklahoma\\ Norman, OK 73019}
\email{victor.bailey@ou.edu}
\author{D. Han}
\address{Department of Mathematics\\
University of Central Florida\\ Orlando, FL 32816}
\email{deguang.han@ucf.edu}
\author{K. Kornelson}
\address{Department of Mathematics\\
University of Oklahoma\\ Norman, OK 73019}
\email{kkornelson@ou.edu}
\author{D. Larson}
\address{Department of Mathematics\\ Texas A\&M University\\ College Station, TX 77843}
\email{larson@math.tamu.edu}
\author{R. Liu}
\address{School of Mathematical Sciences \\ Nankai University\\ Tianjin, 300071 China}
\email{ruiliu@nankai.edu.cn}
\thanks{D. Han  is partially supported by the NSF grant DMS-2105038.}
\subjclass[2020]{Primary 46B15, 47D03, 42C15, 46N99, 20M15}
\keywords{Dynamical sampling and frames, semigroups, frame representation, invariant and hyperinvariant subspace}


%
\begin{abstract} 
The theory of dynamical frames evolved from practical problems in  dynamical sampling where the  initial state of a vector needs to be recovered from the space-time samples of evolutions of the vector. This leads to  the investigation of structured  frames obtained  from  the orbits of evolution operators. One of the basic problems in dynamical frame theory is to determine the semigroup representations, which we will call {\it central  frame representations},  whose  frame generators are unique (up to equivalence). Recently, Christensen, Hasannasab, and Philipp proved that all frame representations of the semigroup $\Bbb{Z}_{+}$ have this property. Their proof of this result relies on the characterization of the structure of shift-invariant subspaces in $H^2(\mathbb{D})$ due to Beurling. In this paper we settle the general uniqueness problem by  presenting  a characterization of central frame representations for any semigroup  in terms of the co-hyperinvariant subspaces of the left regular representation of the semigroup. This result is not only consistent with the known result of Han-Larson in 2000 for group representation frames, but also proves that all the frame generators of a semigroup generated  by any $k$-tuple $(A_1, ... A_k)$  of commuting bounded linear operators on a separable Hilbert space $H$ are equivalent, a case where the structure of shift-invariant subspaces, or submodules, of the Hardy Space on polydisks $H^{2}(\Bbb{D}^k)$ is still not completely characterized.
\end{abstract}

\maketitle

\section{Introduction}

Dynamical sampling was initially introduced by Aldroubi and his collaborators,   and has been  extensively investigated  over the last decade in the context of frames induced by orbits of operators, often called dynamical frames (c.f.,  \cite{AA-1}-\cite{AA-6}, \cite{Ole-1}-\cite{Ole-5}, \cite{Carlos-1, Carlos-2, Carlos-3}, \cite{Powell}). The basic model of dynamical sampling deals with the problem of recovering (or determining) a  signal/image/state vector $f$ in a Hilbert space $H$ from some observations  $\{\langle A_t(f) , g\rangle: t\in \mathcal{T}, g\in \mathcal{G}\}$ of the time  evolutions of the initial state of the vector $f$, 
  where $A_t$ is the evolution (usually, bounded and linear) operator at time $t$ and $\mathcal{G}$ is a fixed (usually, singleton or finite) subset of $H$. The core questions in dynamical sampling ask to find the conditions on $A_t (t\in \mathcal{T})$ and $\mathcal{G}$ such that every $f\in H$ can be recovered or uniquely determined in a stable way by  the observations $\{\langle f , (A_t)^*g\rangle: t\in \mathcal{T}, g\in \mathcal{G}\}$, where $A^*$ is the adjoint operator of $A$. While the unique determination of $f$ requires the completeness property of  the system $\{A_t^*g: t\in \mathcal{T}, g\in \mathcal{G}\}$, the stable recovery  usually requires the ``frame property"  of  the system $\{A_t^*g: t\in \mathcal{T}, g\in \mathcal{G}\}$.
  
 Recall that a  {\it  frame} for a separable Hilbert space, $H$ is a countable collection of vectors $\{f_n\}_{n\in I} \subset H$ that satisfy the following {\it frame  inequality} for every $f \in H$, 
 \begin{displaymath} C_{1} \|f\|^2 \leq \sum_{n \in I} | \langle f , f_n \rangle |^2 \leq C_{2} \|f\|^2 \quad 
\end{displaymath} 
where $ 0 < C_{1} \leq C_{2}$ are the (upper and lower) frame bounds. A {\it Parseval frame} is a frame for which  $C_1 = C_2 = 1$ , and a {\it Bessel sequence}  is a collection of vectors satisfying the upper bound in the frame inequality.  Given a Bessel sequence $\{f_n\}_{n \in I}$, the {\it analysis operator} is the bounded linear operator  $\Theta: H \to \ell^2(I)$ defined by $$ \Theta(f) = \{{\langle f, f_n \rangle}\}_{n \in I} \ \ ( \forall   f\in H )$$ and its adjoint $\Theta^*$ is the  {\it synthesis operator}. 
In the case that $\{f_n\}_{n \in I}$  is a frame, $S = \Theta^*\Theta$ is  a bounded invertible  operator that is called the  {\it frame operator}. One of the important concepts in frame theory is that of the concept of equivalent frames in different settings. In this paper we use the definition that a frame $\{f_n\}_{n\in I}$ for a Hilbert space $H$ is {\it equivalent} to a frame $\{g_n\}_{n\in I}$ for a Hilbert space $K$ if there exists a bounded linear operator $T: H\to K$ such that $Tf_n = g_n$ for every $n\in I$. It is well known (c.f.  \cite{HL-MAMS}) that two frames are equivalent if and only if the range spaces of their analysis operators are the same.
  
  In addition to the frame property, further structures on $\{A_t: t\in\mathcal{T}\}$  are also imposed in dynamical sampling. A typical and well-studied model in the literature is the semigroup structure  $\{A^n: n\geq 0\}$ generated by a single operator $A\in B(H)$. In this case, Christensen, Hasannasab and Philipp \cite{Ole-3}  obtained the uniqueness property for the frame generators of the system, i.e., any two dynamical frames $\{A^n\xi: n\geq 0\}$ and $\{A^n\eta: n\geq 0\}$  for a separable infinite-dimensional Hilbert space $H$ are  equivalent, i.e. there exists an invertible $ T \in B(H)$   that commutes with $A$ and  $\eta = T\xi$. Recently, there has been great interest in establishing the dynamical frames generated by the orbits of several  operators, such as frames generated by the orbit  $\{A_1^{n_1}A_{2}^{n_2}\cdots A_{k}^{n_k}:  n_i\geq 0\}$, where $(A_1, ... , A_k)$ is a $k$-tuple of bounded linear operators on a separable Hilbert space $H$. All of these frames belong to a large class of structured frames, namely, semigroup representation frames. So a basic and natural question asks how much of the frame theory by the single-operator-generated orbit  $\{A^n: n\geq 0\}$ can be generalized to more general semigroup dynamical frames. In particular, the question of whether the frames generated by every frame  representation of any commutative semigroup are equivalent remains open. 
    
The model spaces to study the dynamical frames of the form $\{A^n\xi: n\geq 0\}$ are the  co-invariant (that is, the orthogonal complement of an invariant subspace) subspaces of the unilateral shift operator (that is, the operator $T_z \in B( H^2(\mathbb{D}))$ such that $T_zf(z) = zf(z)$ $\forall f \in H^2(\mathbb{D}))$ on the Hardy space $H^2(\Bbb{D})$  where $A$ is the compression of unilateral shift operator  to the co-invariant subspace. From this it is not surprising that the proof of the equivalence of singly-generated dynamical frames mentioned above relies on the characterization of the structure of shift-invariant subspaces of the Hardy space $H^2(\Bbb{D})$ in terms of  inner functions by Beurling \cite{B48} However, there is no complete characterization of the structure of the shift-invariant subspaces (or submodules) in the higher dimensional cases due to the complexity of the submodules of  the Hardy space over polydisks $H^2(\Bbb{D}^k)$. . Therefore  the  same  approach does not apply to the dynamical frames  of  the form $\{A_1^{n_1}A_{2}^{n_2}\cdots A_{k}^{n_k}\xi:  n_i\geq 0\}$ unless additional properties on the dynamical frames  or on the submodules considered are imposed. 

Motivated by work of Han and Larson on the theory of group representation frames developed in \cite{HL-MAMS},  we use a different approach to address the  above frame generator equivalence problem for frame representations  of arbitrary (not necessarily commutative) semigroups. In this case, our model spaces of dynamical frames  generated by a semigroup $\mathcal{S}$ are  the co-invariant subspaces of the left regular representation of the semigroup. We obtain a complete characterization of all the frame representations with the frame generator equivalence property in terms of  co-hyperinvariant subspaces of the left regular representation of the semigroup. As a consequence of our characterization,  all the dynamical frames generated by  every semigroup representation of $\mathcal{S}$ are equivalent if  the weak operator topology  (wot) closed algebra generated by the left regular representation of $\mathcal{S}$ is maximal abelian.  With the help of this consequence and the characterization of the multiplier algebra of the Hardy space over the  polydisk $\Bbb{D}^k$, we are able to prove that if $(A_1, ... , A_k)$ is a commuting $k$-tuple of bounded linear operators on a separable Hilbert space $H$, then  all the frames of the form $\{A_1^{n_1}A_{2}^{n_2}\cdots A_{k}^{n_k}\xi:  n_j\geq 0\}$ for  $H$ are equivalent. We also discuss several variations of this result including the finite-dimensional frame  representations and the hybrid case  of  $\mathcal{G} \times \mathcal{S}$ for any finite abelian $\mathcal{G}$ and some commutative semigroup $\mathcal{S}$.

\section{Central Frame Representations}

In what follows we will always assume that the  semigroup $\mathcal{S}$ is at most countable,   unital with the identity element $e$,  and has the {\it  left cancellation property} in the sense that  $\forall s, t, \text{and} \, \, z \in \mathcal{S}$, whenever $st =sz$ we have $t = z$. A semigroup group representation on a Hilbert space $H$  refers to a semigroup homomorphism $\pi: \mathcal{S} \to B(H)$ such that $\pi(e) = I$, where $B(H)$ is the space of all the bounded linear operators on the Hilbert space $H$, and $I$ is the identity operator.

\begin{defi} A semigroup representation $\pi: \mathcal{S} \to B(H)$ is called a {\it frame representation} if there exists a vector $\xi\in H$ such that $\{\pi(s)\xi\}_{s\in\mathcal{S}}$ is a frame for $H$. 
In this case we say that $\xi$ is a {\it frame vector} or frame generator of $\pi$. We also call such a frame a {\it dynamical frame} or {\it semigroup representation frame}. 
\end{defi} 
Similarly, a Bessel vector $\xi$ of  a semigroup representation $\pi: \mathcal{S} \to B(H)$ refers to the case when $\{\pi(s)\xi\}_{s\in\mathcal{S}}$ is a Bessel sequence, and its analysis operator will be denoted by $\Theta_x$. In the case that $\{\pi(s)\xi\}_{s\in\mathcal{S}}$  is  a Parseval  frame (respectively, an orthonormal basis)  for $H$, we say that $\xi$ is a  Parseval frame vector (respectively, a complete wandering vector)  of $\pi$.

Two frame vectors $\xi$ and $\eta$ for a frame representation $\pi$  are said to be {\it equivalent}  if $\{\pi(s)\xi\}_{s\in\mathcal{S}}$  and $\{\pi(s)\xi\}_{s\in\mathcal{S}}$ are equivalent frames, that is,  there exists a bounded invertible linear operator $T\in B(H)$ such that $T\in \{\pi(s): s\in \mathcal{S}\}'$ and $T\xi = \eta$, where $$\mathcal{A}' = \{T \in B(H) : TA = AT, \forall A\in\mathcal{A}\}$$ is the {\it commutant} of $\mathcal{A} \subset B(H)$. 

Let $\Bbb{Z}_{+} = \{0\}\cup \Bbb{N} = \{n\in \Bbb{Z}: n \geq 0\}$ and  $\Bbb{Z}_{+}^{k} =  \{ (n_1, ... , n_k):  n_j \in \Bbb{Z}_{+}\}$. Then it is clear that  every representation $\pi$ of $\Bbb{Z}_{+}^{k}$ on a Hilbert space $H$ has the form 
$$
\pi(n_1, ..., n_k) = A_1^{n_1}\cdots A_{k}^{n_k}
$$
where $(A_1, ... , A_k)$ is a commuting $k$-tuple of bounded linear operators on $H$.  Another special case is when $\mathcal{S}$  is a countable group. In this case, the dynamical frames generated by group representations  (or projective unitary representations) belong to  one of the most important classes of structured frames, namely group representation frames  This class includes  well-known examples such as  Gabor frames and wavelet frames.

 In order to classify dynamical frames we formally introduce the following concept:

\begin{defi} A frame representation $\pi: \mathcal{S}\to B(H)$ is called {\it central} if all the frame vectors of $\pi$ are equivalent. 
\end{defi} 
 
With the above definition, The result of Christensen et al. can be restated as follows:

\begin{theo}\label{Ole}  \cite{Ole-3} Every  frame representation $\pi: \Bbb{Z}_{+} \to B(H)$ is  central.
\end{theo} 

 Our first goal  of this paper is to characterize all the central frame representations for arbitrary semigroups. Let $H$ and $K$ be separable Hilbert spaces. Given a semigroup $\mathcal{S}$, we say that two representations $\pi: \mathcal{S}\to B(H)$ and  $\sigma: \mathcal{S}\to B(K)$ are equivalent if there exists a bounded invertible linear operator $T: H\to K$ such that $\pi(s) = T^{-1}\sigma(s)T$ holds for every $s\in\mathcal{S}$. It is straightforward to verify  the following lemma.
 \begin{lem}
 Let $\mathcal{S}$  be a semigroup and let $\pi: \mathcal{S}\to B(H)$ and  $\sigma: \mathcal{S}\to B(K)$ be semigroup representations where $H$ and $K$ are separable Hilbert spaces. Then the following statements hold:
 \begin{itemize}
 
 \item[(i)] If $\pi$ and $\sigma$ are equivalent, then $\pi$ is a frame representation if and only if $\sigma$ is a frame representation. Furthermore,  $\pi$ is central if and only if $\sigma$ is central. 
 
 \item[(ii)]  Two frame representations $\pi$ and $\sigma$ are equivalent if and only if  there exist $\xi\in H$ and $\eta\in K$ such that $\{\pi(s)\xi\}_{s\in\mathcal{S}}$ and $\{\sigma(s)\eta\}_{s\in\mathcal{S}}$ are equivalent frames.
 
 \item[(iii)] If two representations $\pi$ and $\sigma$ of $\mathcal{S}$  both admit complete wondering vectors, then they are unitarily equivalent. 
 
 \end{itemize}
\end{lem}
Recall that the left regular representation of $\mathcal{S}$ is the semigroup homomorphism  $\lambda: \mathcal{S} \to B(\ell^2(\mathcal{S}))$ defined by
$$
\lambda(s) \delta_{t} = \delta_{st}, \ \ \forall t\in \mathcal{S}
$$
where $\delta_s(s) = 1$ and $\delta_{s}(t)= 0$ for $t\neq s$. Since $\mathcal{S}$ is unital and has the left cancellation property, we get that $\lambda(s)$ is an isometry for each $s$, and $\{\lambda(s) \delta_{e}: s\in \mathcal{S}\}$ is the standard orthonormal basis of $\ell^2(\mathcal{S})$. From this, we see that $\lambda$ is a representation that  admits a  complete wandering vector $\delta_e$. Thus, $\lambda$ is unitarily equivalent to any representation that also admits a complete wondering vector. In what follows we will use $\mathcal{A}_{\mathcal{S}}$ to denote the closure of the algebra generated the left regular representation in the weak operator topology, i.e., 
$$
\mathcal{A}_{\mathcal{S}} = \overline{span}^{wot}\{\lambda(s); s\in\mathcal{S}\}.
$$
where $wot$ denotes the weak operator topology.

A closed subspace $M \subseteq H$   is called {\it co-invariant} under a subset   $\mathcal{A}$ $\subseteq$ $B(H)$ if  $M^{\perp}$ is an invariant subspace (that is, $A M^{\perp} \subset M^{\perp}$ for all $A \in  \mathcal{A}$) of $\mathcal{A}$. In this case we also say that the orthogonal projection $P$ onto $M$ is co-invariant  under  $\mathcal{A}$. We say a subspace in $H^2(\Bbb{D}^k)$ is shift-invariant (or a  {\it submodule} of $H^2(\Bbb{D}^k)$) if it is invariant under each of the shift operators $T_{z_i}$ where $T_{z_i}f(z_1, z_2, \ldots, z_k) = z_i f(z_1, z_2, \ldots, z_k)$ \, $\forall f \in H^2(\Bbb{D}^k)$. In the literature, a co-invariant subspace for each of the shift operators (that is, the orthogonal complement of a submodule in $H^2(\Bbb{D}^k))$ is often called a quotient module or a model space. Now let $P$ be the orthogonal projection onto  a co-invariant subspace $M$ of the algebra  $\mathcal{A}_\mathcal{S}$  generated by $\lambda$. Then
$P\lambda(s) = P\lambda(s)P$ for every $s\in \mathcal{S}$. This implies that the  compression of $\lambda$ to the co-invariant subspace $M$  defined by
$$\lambda_{P}(s) := P\lambda(s)P, \ (s\in\mathcal{S})$$
is a  frame representation of  $\mathcal{S}$ with a frame vector $\xi = P\delta_e \in M$. Borrowing the terminology from Hardy space theory, in  the following we will call the subrepresentation $\lambda_P$ a {\it model space representation} of $\mathcal{S}$.

Since  semigroup frame representations (respectively,  central frame representations) are preserved under equivalent representations, the following result tells that we only need to work on the model space representations.

\begin{prop} \label{model}
     Let S be a unital countable semigroup with the left-cancellation property. Suppose that $\pi:  S\to B(H)$ is a frame representation with  frame generator $\xi$. Let $P$ be the orthogonal projection onto the range space $\Theta_{\xi}(H)$\ $\subseteq \ell^2(\mathcal{S})$ of the analysis operator $\Theta_\xi$. Then we have the following

(i) $\Theta_{\xi}(H)$ is co-invariant under the left regular representation $\lambda$.

(ii) $\pi$ is equivalent to the model space  representation $\lambda_P$. Moreover, $\{\pi(s)\xi\}$ and $\{P\lambda(s)P\delta_{e}\}$ are equivalent frames.


\end{prop}

\begin{proof}
(i)  
For all $s,t,w\in S$,
we have
\[\big\langle \lambda(s)^*(\delta_t),\delta_w \big\rangle=\big\langle \delta_t, \lambda(s)(\delta_w)\big\rangle
=\langle \delta_t, \delta_{sw}\rangle=\left \{\begin{array}{ll}
1, & \mbox{ if }\; t=sw\,;\\
0, & \mbox{ otherwise.}
\end{array}
\right.
\]

By the left-cancellation property, we obtain that
\[\lambda(s)^*(\delta_t)= \left \{
\begin{array}{ll}
\delta_w, & \mbox{ if }\; t=sw\,;\\
0, & \mbox{ otherwise.}
\end{array}
\right.\] 

Then for any $x\in H$, we have
\begin{eqnarray*}
  \lambda(s)^*\big(\Theta_{\xi}(x)\big) &=& \lambda(s)^*\Big( \sum_{t\in S}\big\langle x,\pi(t)(\xi) \big\rangle \delta_t
  \Big) = \sum_{t\in S} \big\langle x,\pi(t)(\xi) \big\rangle \lambda(s)^* (\delta_t)\\
  &=& \sum_{w \in S, t=sw} \big\langle x,\pi(t)(\xi) \big\rangle \delta_w
  = \sum_{w\in S} \big\langle x,\pi(sw)(\xi) \big\rangle \delta_w \\
  &=& \sum_{w\in S} \big\langle x,\pi(s)\pi(w)(\xi) \big\rangle \delta_w
  = \sum_{w\in S} \big\langle \pi(s)^*(x), \pi(w)(\xi) \big\rangle
  \delta_w \\
  &=& \Theta_{\xi}\big(\pi(s)^*(x)\big),
\end{eqnarray*}
which implies that $\lambda(s)^*\Theta_{\xi} =\Theta_{\xi} \pi(s)^*$. Thus $\Theta_{\xi}(H)$ is invariant
under $\lambda(s)^*$ for every $s$, and so  $\Theta_{\xi}(H)$ is co-invariant under the left regular representation $\lambda$.

(ii) Since $P$ is co-invariant, we have $P\lambda(s) = P\lambda(s)P$. Thus the range space of the analysis operator of the frame $\{\lambda_{P}(s)P\delta_{e}\}_{s\in \mathcal{S}}  (= \{P\lambda(s)P\delta_{e}\}_{s\in \mathcal{S}})$ is $range(P) = \Theta_{\xi}(H)$. Since the two frames $\{\pi(s)\xi\}_{s\in \mathcal{S}}$ and $\{\lambda_P(s)P\delta_{e}\}_{s\in \mathcal{S}}$ have the same range space of their analysis operators, this implies that they are equivalent. Consequently, $\pi$ and $\lambda_{P}$ are equivalent frame representations.
\end{proof}

In the case that $\mathcal{S}$ is a countable group,  the central frame representations have been completely characterized by Han and Larson in 2000.

\begin{theo} \label{group-case} \cite{HL-MAMS} Let $\pi$ be a frame representation of a countable group $\mathcal{G}$ with a frame vector $\xi$, $P$ be the orthogonal projection onto the range space  $\Theta_{\xi}(H)$ of the analysis operator and $\mathcal{M}$ be  the von Neumann algebra generated by the left regular representation of $\mathcal{G}$
 Then the following are equivalent:

\begin{itemize}

\item[(i)]  $P$ is in the center of $\mathcal{M}$, i.e., $P\in \mathcal{M}\cap \mathcal{M}'$.

\item[(ii)] Every  frame vector $\eta$ of   $\pi$  is equivalent to $\xi$, i.e.,  $\pi$ is central.

\item[(iii)] For every Bessel vector $x$ of $\pi$, there exists an operator $T\in \mathcal{M}'$ such that $x = T\xi$.

\end{itemize}
\end{theo}

The above result is the reason we adopt of the term ``central" for a frame representation with the property that all the frame vectors are equivalent. If $\mathcal{G}$ is an abelian group, then $\mathcal{M} = \mathcal{M}'$. This immediately implies the following:

\begin{coro} \label{abelian-group}  \cite{HL-MAMS} Every frame representation of  a countable group $\mathcal{G}$ is  central  if and only if $\mathcal{G}$ is abelian.
\end{coro}

We remark that with a different approach, Aguilera et al.  also proved (among some other interesting results)  the above corollary for the special case when $\mathcal{S} = \Bbb{Z}\times \Bbb{Z}$ (Theorem 4.7.  \cite{Carlos-1}). 

Unlike the group case where the $wot$-closed algebra of the left regular representation is a self-adjoint  (i.e., $*$-closed ) operator algebra, the $wot$-closed algebra $\mathcal{A}_{\mathcal{S}}$ of the left regular representation $\lambda$ for a semigroup  is not necessarily self-adjoint if $\mathcal{S}$ is not a group.  In order to characterize central frame representations for arbitrary semigroups,  we need to examine the hyperinvariant subspaces of the algebra generated by the left regular representation. The study of hyperinvariant subspaces for  a  single operator has a very rich history with extensive research  activity occurring during the 1970's including a notable breakthrough in 1973  by Lomonosov on the existence of hyperinvariant subspaces for compact operators (see \cite{Douglas, Foias-1, Foias-2, Foias-3} and the references therein). Here we need to establish the existence of and/or provide characterizations for the hyperinvariant subspaces of the semigroup algebra $\mathcal{A}_{\mathcal{S}}$.

\begin{defi}  Let $H$ be a Hilbert space. An invariant subspace $M$ of a subalgebra $\mathcal{A}$ of $B(H)$ is called {\it hyperinvariant} if it is also invariant under $\mathcal{A}'$. We say a subspace $N \subseteq H$ is a  {\it co-hyperinvariant subspace} of $\mathcal{A}$  if $N^{\perp}$ is a hyperinvariant subspace of $\mathcal{A}$. 
\end{defi}

In the case that $\mathcal{A}$ is a von Neumann algebra, a subspace $M$ is hyperinvariant of $\mathcal{A}$ if and only if $P\in \mathcal{A}\cap \mathcal{A}'$, where $P$ is the orthogonal projection onto $M$. Moreover, $M$ is hyperinvariant if and only if $M$ is co-hyperinvariant. So Theorem \ref{group-case} can also be rephrased as: {\it A frame representation $\pi$ of a countable group $\mathcal{G}$ with a frame vector $\xi$ is central  if and only if $\Theta_{\xi}(H)$ is co-hyperinvariant.}  Therefore the main result of this section on the characterization of central frame representations for arbitrary semigroups is consistent with  the group case. To establish this result, we need the following lemma.

\begin{lem} \label{key-lem1} Let $\pi: \mathcal{S} \to B(H)$ be a frame representation of a unital semigroup $\mathcal{S}$ with the left cancellation property. Suppose that $\xi$ is a frame vector for $\pi$ such that $\Theta_{\xi}(H)$ is  co-hyperinvariant under the left regular representation. Then for every Bessel vector $\eta$ for $\pi$, there is a bounded linear operator $A\in \pi(\mathcal{S})'$ such that $\eta = A\xi$.
\end{lem}

\begin{proof} Let $P$ be the orthogonal projection onto the range  $\Theta_{\xi}(H)$ of the analysis operator of the frame representation. By Proposition \ref{model}, $\{\pi(s)\xi\}_{s\in\mathcal{S}}$ and $\{\lambda_{P}(s)\psi\}_{s\in\mathcal{S}}$ are equivalent frames, where  $\psi = P\delta_{e}$. Let $T: H\to \Theta_{\xi}(H)$ be the bounded invertible linear operator such that $T\pi(s)\xi = \lambda_{P}(s)\psi$ for every $s\in\mathcal{S}$. Then for any $s, t\in\mathcal{S}$, we have
\begin{eqnarray*}
T\pi(s)(\pi(t)\xi) &= &T\pi(st)\xi = \lambda_{p}(st)\psi = \lambda_{P}(s)\lambda_{P}(t)\psi \\
&=& \lambda_{P}(s)T(\pi(t)\xi).
\end{eqnarray*}
Thus, $T\pi(s) = \lambda_{P}(s)T$ for all $s\in\mathcal{S}$. Let $\phi = T\eta$. Then $\phi\in \Theta_{\xi}(H)$ and for every $u\in \Theta_{\xi}(H)$ we have
\begin{eqnarray*}
\sum_{s\in \mathcal{S}} |\langle u, \lambda_{P}(s) \phi\rangle |^2 &=& \sum_{s\in \mathcal{S}} |\langle u, \lambda_{P}(s) T\eta\rangle |^2  \\
&=&  \sum_{s\in \mathcal{S}} |\langle u, T\pi(s) \eta\rangle |^2 \leq B||T^*||^2 ||u||^2
\end{eqnarray*}
where $B$ is the Bessel bound for $\{\pi(s)\eta\}_{s\in\mathcal{S}}$. Thus, $\phi$ is a Bessel vector of $\lambda_P$ and consequently of $\lambda$ as well. Define $U: \ell^2(\mathcal{S}) \to \ell^2(\mathcal{S})$ by $U(\delta_s) = \lambda(s)\phi$. Then for all $s, t\in \mathcal{S}$ we have
\begin{eqnarray*}U\lambda(s)(\lambda(t)\delta_e) &= & U\lambda(st)\delta_e= \lambda(st)\phi \\
&= & \lambda(s)\lambda(t)\phi = \lambda(s) U\lambda(t)\delta_e.
\end{eqnarray*}
Thus $U\in \mathcal{A}_{\mathcal{S}}'$.   Define $A = T^{-1}PUPT$. Since $\Theta_{\xi}(H)$ is  co-hyperinvariant, we have $PUP = PU$ and $P\lambda(s) = P\lambda(s)P$ for every $s\in\mathcal{S}$. Thus we get
\begin{eqnarray*}
A\pi(s) &=& T^{-1}PUPT\pi(s) = T^{-1}PUP\lambda_{P}(s)T \\
& = &T^{-1} (PUP)(P\lambda(s)P)T = T^{-1} PU \lambda(s)T\\
&=& T^{-1} P \lambda(s) UT  =  T^{-1} (P \lambda(s)P)( P U P)T\\
&=&T^{-1}\lambda_{P}(s) PUPT = \pi(s)T^{-1}PUPT =\pi(s)A.
\end{eqnarray*}
Thus $A\in \pi(\mathcal{S})'$. Moreover, we also have 
 \begin{eqnarray*}
    A\xi  &= & T^{-1}PUPT\xi = T^{-1}PUP\psi = T^{-1}PUP (P\delta_e)\\
    &=& T^{-1}PU\delta_{e} = T^{-1}P\phi = T^{-1} \phi\\
    &=& T^{-1}(T\eta) = \eta,
 \end{eqnarray*}
which completes the proof.
\end{proof}

\begin{theo} \label{main-thm1} Let $\pi: \mathcal{S} \to B(H)$ be a frame representation of a unital semigroup $\mathcal{S}$ with left cancellation property. Then the following statements are equivalent:

\begin{itemize}

\item[(i)] $\pi$ is a central frame representation.

\item[(ii)] The range space of the analysis operator  $\Theta_\xi$ of every frame vector $\xi$  for $\pi$ is a  co-hyperinvariant subspace of $\mathcal{A}_{\mathcal{S}}$.

\end{itemize}
\end{theo}

\begin{proof}

 $(i)\Rightarrow (ii)$ Assume that  $\pi$ is a central frame representation. Let $\xi$ be a frame generator of $\pi$ and $P$ be the orthogonal projection onto $\Theta_{\xi}(H)$. Then, by Proposition \ref{model}, we know that $\Theta_{\xi}(H)$ is co-invariant under $\mathcal{A}_{\mathcal{S}}$, i.e.,  $PA = PAP$ for every $A\in  \mathcal{A}_{\mathcal{S}}$. We need to show that this also holds for every  $A\in \mathcal{A}_{S}'$. 

Let $A\in \mathcal{A}_{S}'$ be arbitrary. Since $\mathcal{A}_{S}'$ contains the identity operator $I$, by replacing $A$ with $A + c I$ for  some  $c> ||A||$, we can assume that $A$ is invertible.
 Since $A\in \mathcal{A}_{S}'$ is invertible, we get that $\{\lambda(s)A\delta_{e}\}_{s\in\mathcal{S}}$ $= \{A\lambda(s)\delta_{e}\}_{s\in\mathcal{S}}$ is a Riesz basis for $\ell^{2}(\mathcal{S})$ and hence $\{P\lambda(s)A\delta_{e}\}_{s\in\mathcal{S}}  = \{P\lambda(s)PA\delta_{e}\}_{s\in\mathcal{S}} $ is a frame for $P(H)$. Let $\psi = P\delta_{e}$ and $\eta = PA\delta_e$. Then both $\psi$ and $\eta$ are frame vectors for the frame representation $\lambda_{P}$. Since $\pi$ is a central frame representation that is equivalent to $\lambda_P$ by Proposition \ref{model}, we get $\lambda_{P}$ is central. Thus, there exists an invertible operator $T\in \{P\lambda(s)P: s\in S\}'$ such that 
$$
T\psi= \eta.
$$
Define $V = PTP\in B(\ell^2(S))$. Then we have
$$V\delta_e= PTP\delta_e = PT\psi =T\psi = \eta =  PA\delta_e.$$
Thus, from $P\lambda(s) = P\lambda(s)P$ for every $s\in\mathcal{S}$,  we get
\begin{eqnarray*}
P\lambda(s)PV\delta_e &=& (P\lambda(s)P)PA\delta_e = P\lambda(s)PA\delta_e \\
& =& P\lambda(s)A\delta_e = PA\lambda(s)\delta_e\\
& =& PA\delta_s
\end{eqnarray*}
for all $s\in\mathcal{S}$. On the other hand, since $T$ commutes with $\lambda_{P}$, we have 
$$
VP\lambda(s)P = PT\lambda_{P}(s) = \lambda_{P}(s)T = P\lambda(s)PV.
$$
Hence, we also have
\begin{eqnarray*}
P\lambda(s)PV\delta_e &=& VP\lambda(s)P\delta_e = VP\lambda(s)\delta_e \\
&=& V\lambda(s)\delta_e = V\delta_s
\end{eqnarray*}
for all $s\in\mathcal{S}$ where we use the fact that $VP = V$. Therefore, we have proved that $PA\lambda(s)\delta_e = V\delta_s$ for all $s\in\mathcal{S}$, which implies that 
$PA = V$. Hence, we get  $PAP = VP = V =PA$ which implies that $P$ is  co-hyperinvariant as $A\in \mathcal{A}_{S}'$ was arbitrary.

$(ii) \Rightarrow (i)$ Let $\xi$ and $\eta$ be two frame vectors for $\pi$.  Then, by assumption,  both $\Theta_{\xi}$ and $\Theta_{\eta}$ are  co-hyperinvariant subspaces of the left regular representation. Thus, by Lemma \ref{key-lem1}, there exist two operators $A, B\in \pi(\mathcal{S})'$ such that $A\xi = \eta$ and $B\eta = \xi$.
So,  for every $s\in\mathcal{S}$,   we get
    $$
    BA \pi(s) \xi = \pi(s) BA\xi =\pi(s)\xi
    $$
    and 
    $$
    AB \pi(s) \eta = \pi(s) AB\eta =\pi(s)\eta
    $$
  This implies that  $BA = I =AB$. Hence $A$ is invertible and consequently $\{\pi(s)\xi\}_{s \in S}$ and $\{\pi(s)\eta\}_{s \in S}$ are equivalent frames. Therefore $\pi$ is central.
\end{proof}

\begin{rem} In the group case, by Theorem \ref{group-case},   we only need to establish the co-hyperinvariance  of $\Theta_{\xi}$$(H)$ for a  fixed  frame vector $\xi$ to characterize the central frame representation. In that case, the theorem was proved by using the finite von Neumann algebra property of the algebra generated by the left regular representation. However, we do not have the same statement for general semigroup central frame representations. This case  seems to require a von Neumann algebra type of finiteness property  for the non-self-adjoint algebra $\mathcal{A}_{\mathcal{S}}'$.
\end{rem}

The proof of Theorem \ref{main-thm1} clearly implies the following corollaries.

\begin{coro}
Let $\pi: S \to B(H)$ be a frame representation of a countable unital semigroup $S$ with the left-cancellation property, and $\xi, \eta\in H$ be any two frame vectors of $\pi$. If $\Theta_{\xi}(H)$ and $\Theta_{\eta}(H)$ are both  co-hyperinvariant, then $\{\pi(s)\xi\}_{s \in S}$ and $\{\pi(s)\eta\}_{s \in S}$ are equivalent frames.
 \end{coro}

Similarly, from the proofs or with minor modifications of the proofs for  Lemma \ref{key-lem1} and Theorem \ref{main-thm1}, we also obtain the following statement.

\begin{coro} Let $\pi: S \to B(H)$ be a frame representation of a countable unital semigroup $S$ with the left-cancellation property, and $\xi\in H$ be a frame vector of $\pi$. Then the following are equivalent:

    (i) $\Theta_{\xi}(H)$ is co-hyperinvariant under the left regular representation. 

    (ii) $\Theta_{x}(H) \subseteq \Theta_{\xi}(H)$ holds for every Bessel vector $x$ of $\pi$, or equivalently, for every Bessel vector $x$, there exists a unique $A \in \{\pi(S)\}'$
    such that   $x = A\xi$.

    (iii)  $\Theta_{\eta}(H) \subseteq \Theta_{\xi}(H)$ holds for every frame vector $\eta$ for  $\pi$,   or equivalently,  for every frame vector $\eta$, there exists a unique $A \in \{\pi(S)\}'$
    such that   $\eta= A\xi$.
    
\end{coro}

\section{Dynamical frames by commutative semigroups}

With the aim of generalizing Theorem \ref{Ole}  at a minimum to  the case when $\mathcal{S} = \Bbb{Z}_{+}^k$ for $k > 1$, this section will be focused on discussing the central frame representation property for commutative semigroups. 

Recall that an abelian subalgebra $\mathcal{A}$ of $B(H)$ is called {\it maximal abelian} if $\mathcal{A} = \mathcal{A}'$, or equivalently, $\mathcal{A}$ is not a proper subalgebra of any abelian algebras of $B(H)$. Clearly, if the $wot$-closed algebra $\mathcal{A}_{\mathcal{S}}$ generated by the left regular representation $\lambda$  of a commutative semigroup $\mathcal{S}$ is maximal abelian, then every co-invariant subspace of $\lambda$ is co-hyperinvariant. This immediately, by Theorem \ref{main-thm1} and Proposition \ref{model},  implies the following

\begin{prop} \label{prop-abelian}  Let $\mathcal{S}$ be a unital commutative semigroup with the cancellation property such that $\mathcal{A}_{\mathcal{S}}$ is maximal abelian. Then every frame representation of $\mathcal{S}$ is central. 
\end{prop}

 It is well known that every $wot$-closed self-adjoint abelian subalgebra $\mathcal{A}$ of $B(H)$ with a cyclic vector  $\xi$ (i.e., $\overline{span}\mathcal{A}\xi = H$)  is maximal abelian (c.f. Corollary 7.2.16, \cite{Kadison}), and this is no longer true for non-self-adjoint algebras (c.f. Example 1,  \cite{Deddens-PAMS}). However,  we believe that  $\mathcal{A}_{\mathcal{S}}$ is always maximal abelian (and hence every frame representation is central) for every unital commutative semigroup $\mathcal{S}$ with the left cancellation property. In fact we make the following conjecture.

\vspace{3mm}

\begin{conjecture} If $\mathcal{S}$ is a unital commutative semigroup with (or without) the cancellation property, then the $wot$-closed subalgebra $\mathcal{A}$ generated by any  frame representation $\pi$ of $\mathcal{S}$ is maximal abelian.
\end{conjecture}

\vspace{3mm}

The following tells us that for the  finite dimensional representation case, every frame representation of a commutative semigroup is central.

\begin{prop} Let $\mathcal{S}$ be a commutative semigroup. Then every frame representation of $\mathcal{S}$ on a finite dimensional Hilbert space $H$ is central.
\end{prop}
\begin{proof}  Let $\pi: \mathcal{S} \to B(H)$ be a frame representation such that $\dim H < \infty$, and let $\mathcal{A}$ be the wot-closed algebra generated by $\pi(\mathcal{S})$.
Let $\xi$ and $\eta$ be any two frame vectors for $\pi$. Since $H$ is finite dimensional, we get that $$\mathcal{A}\xi =\{A\xi: A\in\mathcal{A}\} = H,$$
i.e., $\xi$ is a strictly cyclic vector of $\mathcal{A}$. So there exists an operator $A\in \mathcal{A}$ such that $A\xi = \eta$. Since $\mathcal{S}$ is commutative, 
we get that $\mathcal{A}$ is abelian and so $A\in \mathcal{A}'$. Note that 
$$
range (A) \supseteq \{A\pi(s)\xi: s\in \mathcal{S} \} = \{\pi(s)A\xi: s\in \mathcal{S} \} =\{\pi(s)\eta: s\in \mathcal{S} \}
$$
and as $span\{\pi(s)\eta: s\in\mathcal{S}\} = H$, we get that  $A$ is  surjective and hence invertible as $H$ is finite dimensional. Thus $\xi$ and $\eta$ are equivalent frame vectors for $\pi$.
\end{proof}

We remark that in the finite dimensional representation case, neither the unital nor the left cancellation property on the semigroup $\mathcal{S}$ is required. 

Next we settle the case when $\mathcal{S} = \Bbb{Z}_{+}^k$. For this case, in order to show that  $\mathcal{A}_{\mathcal{S}}$ is maximal abelian we need to use the characterization of the {\it multiplier algebra} $H^{\infty}(\mathbb{D}^k)$ on the Hardy space $H^2(\mathbb{D}^k)$ over the polydisc $\Bbb{D}^k$. Recall that the Hardy Space $$H^2(\mathbb{D}^k) = \{ f: \mathbb{D}^k \to \mathbb{C}\  | \ f(z) = \underset{n}{\sum} c_{n}z^n \, \, \text{with} \, \,  \sum_{n}|c_n|^2 < \infty\}$$ is
the space of holomorphic functions on $\mathbb{D}^k$, the open polydisc in $\mathbb{C}^k$, with square-summable complex coefficients in the power series representation, where $n = (n_1, ... , n_k)$ runs over all the k-tuples of nonnegative integers and $z^n = z_1^{n_1}\cdots z_{k}^{n_k}$.  Let $M_{z_{j}}: H^2(\mathbb{D}^k) \to H^2(\mathbb{D}^k)$ be the multiplication (or shift) operator defined  by $M_{z_j}(f(z)) = z_jf(z) \ \  (\forall z\in \Bbb{D}^k)$, and let $\mathcal{M}(\Bbb{D}^k)$ be the $wot$-closed subalgebra generated by $M_{z_1}, ... , M_{z_k}$.  The multiplier algebra is given by $\{M_{\phi}: \phi\in L^{\infty}(\mathbb{D}^k) \cap H^2(\mathbb{D}^k)\}$ which can also be  identified with the Banach algebra $H^{\infty}(\mathbb{D}^k)$ of bounded analytic functions on $\Bbb{D}^k$ equipped with the supremum
norm $||\phi||_{\infty}$  ($= ||M_{\phi}||$). It is known (c.f.  \cite{Ball, Sarkar}) that $\mathcal{M}(\Bbb{D}^k)' =  H^{\infty}(\mathbb{D}^k)$. We need the following lemma.

\begin{lem} \label{key-lem2} The subalgebra $\mathcal{M}(\Bbb{D}^k)$ is maximal abelian.
 \end{lem}
 \begin{proof} This proof follows the similar argument from the proof of  Lemma 3.1 in  \cite{Mc-TAMS}. Let $A \in \mathcal{M}(\Bbb{D}^k)'$. Since $\mathcal{M}(\Bbb{D}^k)' =  H^{\infty}(\mathbb{D}^k)$, there exists  $\phi\in H^{\infty}(\mathbb{D}^k)$ such that  $A = M_{\phi}$, where $M_{\phi} \in B(H^2(\mathbb{D}^k))$ is the operator defined by multiplication by $\phi$.  Let $$\phi = \sum_{n=0}^{\infty}\phi_n$$ be an  expansion of homogeneous analytic functions  $\phi_n$ on  $\Bbb{D}^{k}$, and let $$\psi_n = {1\over n+1}\sum_{k=0}^{n}\sum_{j=0}^{k} \phi_j.$$  
Then the same Fejer  kernel argument as in Lemma 3.1 in \cite{Mc-TAMS} implies that  $||M_{\psi_{n}}|| \leq ||M_{\phi} ||$. Since $\{M_{\psi_n}f\}_{n=0}^{\infty}$ converges to $M_{\phi}f$ on a dense subspace (the subspace consisting of the polynomials) of $H^2(\Bbb{D}^k)$ and $\{M_{\psi_n}\}_{n=0}^{\infty}$ is uniformly bounded,  we get that $\{M_{\psi_n}\}_{n=0}^{\infty}$ converges to $M_{\phi}$ in the strong operator topology. Since $M_{\psi_n} \in \mathcal{M}(\Bbb{D}^k)$ and $\mathcal{M}(\Bbb{D}^k)$ is also closed in the strong operator topology, we get that  $A = M_{\phi}\in\mathcal{M}(\Bbb{D}^k)$. Thus $\mathcal{M}(\Bbb{D}^k) = \mathcal{M}(\Bbb{D}^k)'$, and hence $\mathcal{M}(\Bbb{D}^k)$ is maximal abelian.
  \end{proof}

 \begin{theo} Every frame representation of $ \Bbb{Z}_{+}^k$ is a central frame representation. Or,  equivalently, if $(A_1, ... , A_k)$ is a commuting $k$-tuple of operators in $B(H)$ and 
  $\{A_1^{n_1}\cdots A_k^{n_k}\xi: n_j\geq 0\}$ and $\{A_1^{n_1}\cdots A_k^{n_k}\eta: n_j\geq 0\}$ are two frames for $H$, there there exists an operator $T\in B(H)$ commuting with each $A_j$ such that $\eta = T\xi$.
\end{theo}
\begin{proof} Define the semigroup representation $\sigma: \Bbb{Z}_{+}^k \to B(H^2(\mathbb{D}))$ by
$$
\sigma(n)  = M_{z_1}^{n_1}\cdots M_{z_k}^{n_k}, \ \  n\in \Bbb{Z}_{+}^k.
$$
Then the constant function $1$ is a complete wandering vector for $\sigma$.  Since both $\lambda$ and $\sigma$ admit complete wandering vectors, we get that  $\sigma$ is equivalent to the left regular representation $\lambda$ of the semigroup $\Bbb{Z}_{+}^k$. Thus $\mathcal{A}_{\Bbb{Z}_{+}^k}$ is maximal abelian if and only if $\mathcal{M}(\Bbb{D}^k)$ is maximal abelian. By Lemma \ref{key-lem2}, we get that  $\mathcal{A}_{\Bbb{Z}_{+}^k}$ is maximal abelian, and therefore,  by Proposition \ref{prop-abelian}, every frame representation of $\Bbb{Z}_{+}^k$ is central.
\end{proof} 

 We say that an element  $s$ in a unital semigroup $\mathcal{S}$ has order $N$ if $s^N = e$ and $s^n \neq e$ for any $0< n < N$.  For the case $N = \infty$, it means $s^n \neq e$ for any $n > 0$. The following addresses the hybrid case  $\mathcal{S} = \mathcal{G}\times \Bbb{Z}_{+}^k$  where $\mathcal{G}$ is a finite abelian group (which seems to cover many interesting cases of the finitely generated commutative semigroups). 

\begin{theo} Let $\mathcal{S}$ be a untial abelian semigroup generated by $s_1, ... s_k$ of order $N_1, ..., N_k$. Suppose that every element $s$ in $\mathcal{S}$ has a unique representation in the sense that 
$$
s_1^{m_1}\cdots s_k^{m_k} = s_1^{n_1}\cdots s_k^{n_k}
$$
implies that $m_i = n_i$ for $i=1, .., k$, where $0\leq m_i, n_i < N_i$.
Then $\mathcal{A}_S$ is maximal abelian, and hence every frame representation of $S$ is central.
\end{theo} 

\begin{proof}  We can assume that $N_1, ... , N_{\ell}$ are finite and $N_{\ell+1} = ... = N_k = \infty$. Let $\mathcal{G}$ be the semigroup generated by $s_1, ... , s_{\ell}$. Then $\mathcal{G}$ is a finite abelian group. By the assumption on the unique representation of every $s$ by the generators $s_1, ... , s_k$, we get that $\mathcal{S}$ is  isomorphic (as semigroups) to $\mathcal{T}: = \mathcal{G}\times \Bbb{Z}_{+}^{m}$, where $m = k-\ell$. Thus we only need to show that $\mathcal{A}_{\mathcal{T}}$ is maximal abelian. Note that   $\mathcal{A}_{\mathcal{T}} = \mathcal{A}_{\mathcal{G}} \otimes  \mathcal{A}_{\Bbb{Z}_{+}^m}$.

Since $\mathcal{A}_{\mathcal{G}}$ is the maximal abelian von Neumann algebra acting on the finite dimension space $\ell^2(\mathcal{G})$ and has a cyclic vector, it is well known that  there exists a unitary operator $U:  \Bbb{C}^{n}\to \ell^{2}(\mathcal{G})$ (where $n = |G|  < \infty$) such that 
$$
U^*\mathcal{A}_{\mathcal{G}} U= \mathcal{D}_{n},
$$
where $\mathcal{D}_{n}$ is the diagonal subalgebra of the matrix algebra $M_{n\times n}(\Bbb{C})$. 

By Lemma \ref{key-lem2}, we know that  $\mathcal{A}_{\Bbb{Z}_{+}^m}$ is maximal abelian, which implies that 
$$
\mathcal{D}_{n} \otimes  \mathcal{A}_{\Bbb{Z}_{+}^m} =  \mathcal{A}_{\Bbb{Z}_{+}^m}\oplus \cdots  \oplus \mathcal{A}_{\Bbb{Z}_{+}^m} \ \  (n-\text{copies})
$$
is maximal abelian. Thus by the unitary equivalence we get that  $\mathcal{A}_{\mathcal{T}}$ is maximal abelian, and therefore from Proposition \ref{prop-abelian} we get that every frame representation of $\mathcal{S}$ is central.
\end{proof}

\vspace{3mm}




\end{document}